\pdfoutput=1
\documentclass[11pt, oneside]{article}

\usepackage{amsmath, amssymb}
\usepackage{geometry}
\usepackage{graphicx}
\usepackage{enumitem}

\newtheorem{theorem}{Theorem}
\newtheorem{proposition}[theorem]{Proposition}
\newtheorem{lemma}[theorem]{Lemma}

\newtheorem{conjecture}[theorem]{Conjecture}
\newtheorem{claim}{\emph{Claim}}
\newtheorem{obs}[claim]{\emph{Observation}}

\newenvironment{proof}{{\noindent \textbf{\textit{Proof.}}}}
{\hfill $\Box$\vspace*{0.1in}}
\newenvironment{proofc}{{\noindent \textbf{\textit{Proof of Claim.}}}}
{\hfill $\Box$\vspace*{0.1in}}
\newenvironment{proofcn}[1]{{\noindent \textbf{\textit{Proof of Claim \ref{#1}.}}}}
{\hfill $\Box$\vspace*{0.1in}}

\newcounter{tmptheorem}

\newcommand{\plus}{{\scriptscriptstyle+}}
\newcommand{\minus}{{\scriptscriptstyle-}}

\long\def\symbolfootnote[#1]#2{\begingroup\def\thefootnote{\fnsymbol{footnote}}
\footnote[#1]{#2}\endgroup}

\begin{document}

\title{Kempe equivalence of edge-colourings in subcubic and subquartic graphs}
\author{
Jessica McDonald\footnote{jessica\textunderscore mcdonald@sfu.ca (corresponding author)}\\
Bojan Mohar \footnote{mohar@sfu.ca. Supported in part by an NSERC Discovery Grant (Canada), by the Canada Research Chair program, and by the Research Grant P1--0297 of ARRS (Slovenia). On leave from: IMFM \& FMF, Department of Mathematics, University of Ljubljana, Ljubljana, Slovenia.}\\
Diego Scheide \footnote{dscheide@sfu.ca}\\
\medskip\\
Department of Mathematics\\
Simon Fraser University\\
Burnaby, B.C., Canada V5A 1S6
}
\date{May 12, 2010}

\maketitle

\bigskip

\begin{abstract}
It is proved that all 4-edge-colourings of a (sub)cubic graph are Kempe equivalent. This resolves a conjecture of the second author. In fact, it is found that the maximum degree $\Delta=3$ is a threshold for Kempe equivalence of $(\Delta+1)$-edge-colourings, as such an equivalence does not hold in general when $\Delta=4$. One extra colour allows a similar result in this latter case however, namely, when $\Delta\leq 4$ it is shown that all $(\Delta+2)$-edge-colourings are Kempe equivalent.
\end{abstract}



\begin{section}{Introduction}

Let $\phi$ be a \emph{$k$-colouring} of a graph $G$, that is, an assignment of the colours $\{1,\ldots, k\}$ to the vertices of $G$ such that adjacent vertices receive different colours. Any pair of colours $a,b \in \{1,\ldots, k\}$ induces a subgraph $G(a,b)$ of $G$, and note that switching $a$ and $b$ on a connected component of $G(a,b)$ results in another $k$-colouring of $G$. We call any such switch a \emph{Kempe change} (or \emph{$K$-change}).  If a $k$-colouring $\psi$ can be obtained from $\phi$ by a sequence of $K$-changes, then we say that $\phi$ and $\psi$ are \emph{Kempe equivalent}, and we write $\phi \sim_k \psi$.  If two colourings differ only by a permutation of colours then they are clearly Kempe equivalent. Therefore, when we are interested in Kempe equivalence, we may as well consider colourings that differ only by a permutation of colours to be the same. In general, we use $\kappa(G, k)$ to denote the number of Kempe equivalence classes of $k$-colourings of $G$. It is worth noting that there exist graphs such $\kappa(G,k)=1$ but $\kappa(G, k+1)>1$ (see \cite{Mo}).

Kempe changes have been a widely used technique in graph colouring theory ever since having been introduced by Kempe in his erroneous proof of the four colour theorem. In recent years however, there has been increased interest in the parameter $\kappa(G,k)$ due to direct applications in approximate counting and applications in statistical physics. In particular, the dynamics of the Wang-Swendsen-Koteck\'{y} algorithm (a Monte Carlo algorithm for the antiferromagnetic $k$-state Potts model) is ergodic at zero-temperature if and only if $\kappa(G, k)=1$ for a corresponding graph $G$  (see eg. \cite{MS}, \cite{WSK}). There has been substantial work on this problem (see eg. \cite{BH}, \cite{FS}), but much remains to be known in the case where the graph in question is a line graph. Colourings of a line graph $L(G)$ correspond to edge-colourings of the graph $G$, and Kempe changes correspond to swapping a pair of colours along either a maximal alternating path or alternating cycle of $G$. Let $\kappa_E(G, k):=\kappa(L(G), k)$, that is, let $\kappa_E(G, k)$ denote the number of Kempe equivalence classes of $k$-edge-colourings of $G$. In \cite{Mo}, the second author proved the following result about Kempe equivalence of edge-colourings. This result involves \emph{chromatic index} $\chi'(G)$, the minimum number of colours needed to edge-colour a graph $G$.

\begin{theorem}\label{thm:mohar}\emph{\cite{Mo}}
  Let $G$ be a graph with maximum degree $\Delta$. If $k\geq \chi'(G)+2$ is an integer, then $\kappa_E(G, k)=1$.
\end{theorem}

Throughout this paper we shall use $\Delta$ to denote the maximum degree of the graph in question. Vizing's Theorem \cite{Vi} says that for any graph $G$, $\chi'(G)$ is equal to either $\Delta$ or $\Delta+1$, and we often label graphs as \emph{class 1} in the former case, and \emph{class 2} in the latter. For class 1 graphs, Theorem \ref{thm:mohar} tells us that $\kappa_E(G, \Delta+2)=1$, while it only says that $\kappa_E(G, \Delta+3)=1$ for class 2 graphs.  One natural question is then to ask whether or not  $\kappa_E(G, \Delta+2)=1$ for all graphs $G$. This question remains open, and no conjecture has been made in either direction, although we do provide the following positive result for subquartic graphs.

\newcounter{thmsubquartic}
\setcounter{thmsubquartic}{\value{theorem}}
\begin{theorem}\label{thm:subquartic}
  If $G$ is a graph with maximum degree $\Delta\leq 4$, then
$\kappa_E(G, \Delta+2)=1$.
\end{theorem}

Our main work here is actually about $(\Delta+1)$-edge-colourings in the special case when $\Delta \leq 3$. While class 1 graphs have $\Delta$-edge-colourings, it is known that $\kappa_E(G, \Delta)>1$ for many class 1 graphs, including $K_{p,p}$ when $p$ is prime. However, in \cite{Mo}, the second author wondered if it would be possible to extend Theorem \ref{thm:mohar} so far as to prove $\kappa_E(G, \Delta+1)=1$ for all graphs, and conjectured that this should be possible if $G$ is subcubic. In this paper we confirm this conjecture.

\newcounter{thmsubcubic}
\setcounter{thmsubcubic}{\value{theorem}}
\begin{theorem}\label{thm:subcubic}
  If $G$ is a graph with maximum degree $\Delta\leq 3$, then all $(\Delta+1)$-edge-colourings of $G$ are Kempe equivalent.
\end{theorem}

In fact, we find that $\Delta=3$ is a threshold for Kempe equivalence of $(\Delta+1)$-edge-colourings, as the following example demonstrates.

\begin{proposition}\label{prop:k5}
  $\kappa_E(K_5, 5)=6$.
\end{proposition}

\begin{proof}
Since the maximum size of a colour class in $K_5$ is two, and since $|E(K_5)|=10$, every 5-edge-colouring of $K_5$ has exactly five colour classes of size two. Hence, given any such colouring $\phi$, each colour of $\phi$ is missing at exactly one vertex. Given any pair of colours in $\phi$, the graph induced by these colours has four edges, and so must be a Hamilton path of $K_5$. Consequently, there are no non-trivial Kempe exchanges that can be performed on $\phi$. Hence, $\kappa_E(K_5, 5)$ is equal to the number of different 5-edge-colourings of $K_5$ (up to permutation of colour classes).

To count the number of distinct 5-edge-colourings of $K_5$, label the vertices $v_1, \ldots, v_5$. Consider the colour class containing the edge $v_1v_2$, say colour 1. There are three choices for the second edge of colour 1: it could appear on $v_3 v_4$, $v_4 v_5$, or $v_3 v_5$. Suppose first that it appears on $v_3 v_4$. Then all other colours, say 2, 3, 4, 5, are incident to $v_5$. Without loss of generality, we may assume that $v_1 v_5, v_2 v_5, v_3 v_5$ and $v_4 v_5$ are coloured 2, 3, 4, 5, respectively. Then the edge $v_2 v_3$ receives either colour 2 or colour 5 (see Figure \ref{fig:k5}). Each choice forces the rest of the colouring. So, there are exactly two different 5-edge-colourings of $K_5$ formed in this way. Our argument is completely symmetric for the other two possible choices of colour class 1. Hence, $\kappa_E(K_5, 5)=6$.
\end{proof}

\begin{figure}[htbp]
  \centering
  \includegraphics[height=4cm]{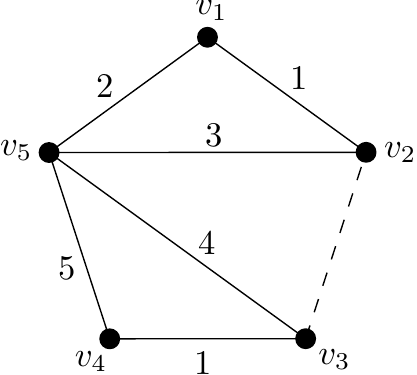}
  \caption{The edge $v_2 v_3$ receives either colour 2 or colour 5}
  \label{fig:k5}
\end{figure}

The above example can be generalized to an infinite family of graphs with $K_E(G, \Delta+1)>1$, namely the set of all graphs $K_{2p-1}$ where $p\geq 3$ is a prime number. This fact is related to the following conjecture of Kotzig, known as the Perfect 1-Factorization Conjecture.

\begin{conjecture}\label{conj:kotzig}\emph{\cite{Ko}}
  For any $n \geq 2$, $K_{2n}$ can be decomposed into $2n-1$ perfect matchings such that the union of any two of these matchings forms a hamiltonian cycle of $K_{2n}$.
\end{conjecture}

This conjecture remains open in general, but Kotzig himself proved it to be true when $n=p$ is prime \cite{Ko}. In this case, by deleting one vertex, we obtain a $(2p-1)$-edge-colouring of $K_{2p-1}$ where the union of any two colour classes forms a Hamilton path, and hence there are no nontrivial Kempe changes. When $p\geq 3$, $K_{2p-1}$  has  a pair of colourings which are different (i.e.\ not a permutation of one another) because $K_5\subseteq K_{2p-1}$, hence we get $\kappa_E(K_{2p-1}, 2p-1)>1$. Note that if the perfect 1-factorization conjecture were proved for all $n$, then we would have an example of a graph $G$ with $\kappa_E(G, \Delta+1)>1$ for all odd values of $\Delta$. As it is, the values $\Delta=2p-2$ for all primes $p\geq 3$ still provide an infinite family of  such examples.

\end{section}

\begin{section}{Proofs of Theorems \ref{thm:subquartic} and \ref{thm:subcubic}}

If $\phi$ is an edge-colouring of a graph $G$, and $a$ and $b$ are two colours of $\phi$, then a path of $G$ with edges coloured alternately $a$ and $b$ under $\phi$ is called an \emph{$(a,b)$-alternating path} in $\phi$. Alternating paths will be very useful in this section, where we prove a series of lemmas leading up to the proofs of Theorems \ref{thm:subquartic} and \ref{thm:subcubic}.

\begin{lemma}\label{lem:path}
  Let $G$ be a graph and suppose that $\phi$ and $\psi$ are $k$-edge-colourings of $G$. Let $F=\{e\in E(G)\,|\,\phi(e)\neq \psi(e)\}$. If $F$ is the union of vertex-disjoint paths in $G$, then $\phi \sim_k \psi$.
\end{lemma}

\begin{proof}
Our proof is by induction on $|F|$. Let $P=v_1 v_2\cdots v_l$ be a maximal path on which $\phi$ and $\psi$ do not agree. Suppose that the edge $v_1v_2$ has  colour 1 under $\phi$, and colour 2 under $\psi$. Since every other edge incident to $v_1$ has the same colour under both $\phi$ and $\psi$, we know that 2 is missing at $v_1$ in $\phi$, and 1 is missing at $v_1$ in $\psi$. We now consider the maximal $(1,2)$-alternating paths starting at $v_1$ in both $\phi$ and $\psi$; call these paths $P_{\phi}$ and $P_{\psi}$, respectively. Since $\phi$ and $\psi$ agree on every edge incident to $P$, there exists a vertex $v_i \in V(P)$, where $2\leq i\leq l$, such that $v_1v_2 \cdots v_i$ is equal to one of $P_{\phi}$ and $P_{\psi}$ and is a subpath of the other. Without loss of generality, we may assume that $P_{\phi}$ is  $v_1v_2 \ldots v_i$. Then we may modify $\phi$ by performing a $(1, 2)$-exchange on $P_{\phi}=v_1v_2 \cdots v_i$. After this exchange, the two colourings agree on the edges of  $v_1v_2 \cdots v_i$. We can thus complete the proof by induction.
\end{proof}

For both Theorem \ref{thm:subquartic} and Theorem \ref{thm:subcubic}, the most difficult cases to handle are regular graphs. Using Lemma \ref{lem:path}, it is actually not hard to prove that all $(\Delta+1)$-edge-colourings  are Kempe equivalent when the graph in question is a connected non-regular subcubic graph, and we can even show this to be true for a special family of non-regular subquartic graphs. In what follows, for $\phi$ an edge-colouring of a graph $G$, and $G'$ a subgraph of $G$, we denote by $\phi|_{G'}$ the restriction of $\phi$ to $G'$.

\begin{lemma}\label{lem:3nonregular}
  If $G$ is a connected subcubic graph that has a vertex of degree at most two, then $ \kappa_E(G,4)=1.$
\end{lemma}

\begin{proof}
The proof is by induction on $|V(G)|$. Let $\phi$ and $\psi$ be 4-edge-colourings of $G$. Choose $v \in V(G)$ with $\deg(v)<3$, and apply induction to the connected components of $G'=G\setminus v$ to get $\phi|_{G'} \sim_4 \psi|_{G'}$. We claim that this series of $K$-changes in $G'$ extends to a series of $K$-changes in $G$.

Suppose that some $K$-change in $G'$ (say a (1,2)-exchange) does not correspond to a $K$-change in $G$. This means that there is a maximal $(1,2)$-alternating path which ends at a neighbour $x$ of $v$ in $G'$, but which continues though $v$, and goes through the second neighbour $y$ of $v$ in $G$. Since $v$ has degree at most $2$, we know that colours 3 and 4 are both missing at $v$, and at least one of these colours is missing at $x$. Assume, without loss of generality, that the edge $xv$ is coloured 1 and that colour 4 is missing at $x$. See Figure \ref{fig:3nonregular} (note that in this figure, and all those that follow, a bar over a colour indicates that the colour is missing at the associated vertex). Then the $(1,2)$-exchange in $G'$ can be extended to a $(1, 2)$-exchange in $G$, provided it is preceded in $G$ by a $(1,4)$-exchange on the edge $xv$.

\begin{figure}[htbp]
  \centering
  \includegraphics[height=2.5cm]{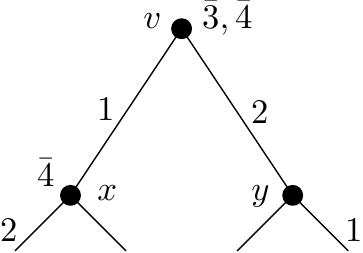}
  \caption{Extending a $K$-change in $G'$ to a $K$-change in $G$ (Lemma \ref{lem:3nonregular})}
  \label{fig:3nonregular}
\end{figure}

The above argument implies that there is a series of $K$-changes which can be performed on $\phi$ so that it agrees with $\psi$ on all edges of $G$, except possibly on the edges incident with $v$. However, since any disagreement will be a path of length one or two, we may apply Lemma \ref{lem:path} to complete the proof.
\end{proof}

\begin{lemma}\label{lem:4nonregular}
  Let $G$ be a connected subquartic graph with the property that no pair of vertices of degree four are adjacent. Suppose further that there exists $v \in V(G)$ satisfying either
  \begin{enumerate}[label=\textup{(\alph*)}]
  \item\label{sublem:4nonregular:d<=2}
    $\deg(v) \leq 2$, or
  \item\label{sublem:4nonregular:d=3}
    $\deg(v)=3$, and $v$ has at most two neighbours of degree four.
  \end{enumerate}
  Then $\kappa_E(G, 5)=1$.
\end{lemma}

\begin{proof}
The proof is by induction on $|V(G)|$. Let $\phi$ and $\psi$ be 5-edge-colourings of $G$. Choose $v \in V(G)$ satisfying one of \ref{sublem:4nonregular:d<=2} or \ref{sublem:4nonregular:d=3}, and let $G'=G\setminus v$. Note that if $u$ is a neighbour of $v$ in $G$ and $\deg(u) \leq 3$ in $G$, then $u$ satisfies \ref{sublem:4nonregular:d<=2} in $G'$. On the other hand, if $\deg(u)=4$ in $G$, then $u$ cannot have any degree $4$ neighbours in $G$, and hence satisfies \ref{sublem:4nonregular:d=3} in $G'$.  Since this holds for every neighbour $u$ of $v$, we may apply induction to the connected components of $G'$ to get $\phi|_{G'} \sim_5 \psi|_{G'}$. We claim that this series of $K$-changes in $G'$ extends to a series of $K$-changes in $G$.

Suppose that some $K$-change in $G'$ (say a (1,2)-exchange) does not correspond to a $K$-change in $G$. This means that there is a maximal $(1,2)$-alternating path that has an end at a neighbour $x$ of $v$ in $G'$, but that goes through $v$, and through a second neighbour $y$ of $v$ in $G$.  Assume, without loss of generality, that the edge $xv$ is coloured 1 and the edge $yv$ is coloured 2. If a colour $\alpha\in\{3,4,5\}$ is missing at both $v$ and $x$, or at both $v$ and $y$, then the $(1,2)$-exchange in $G'$ can be extended to a $(1,2)$-exchange in $G$ provided it is preceded in $G$ by a $(1,\alpha)$-exchange on $xv$ or a $(2,\alpha)$-exchange on $yv$, respectively. Since $x$ and $y$ have degree at most $4$, they are each missing at least one colour  from $\{3,4,5\}$. Without loss of generality, we may assume that $3$ is missing at both $x$ and $y$ but is present at $v$ (say on an edge $vz$), and that $4,5$ are both missing at $v$ but present at $x,y$ (see Figure \ref{fig:4nonregular1}). In particular, this implies that $v$ has degree $3$, and that both $x$ and $y$ have degree $4$ in $G$. Since $v$ satisfies \ref{sublem:4nonregular:d=3}, this means that $z$ has degree at most 3 in $G$. If either 4 or 5 are missing at $z$, then  we can recolour $vz$ with 4 or 5, and then do a $(3, 1)$-exchange on $xv$ to break the $(1, 2)$-path. So, 4 and 5 are both present at $z$, and 1 and 2 are both missing. Hence, we can do a $(2, 3)$-exchange on the path $yvz$, which has the effect of breaking the $(1, 2)$-path, since 1 is missing at $z$. Hence, we have shown that there is a series of K-changes which can be performed on $\phi$ so that it agrees with $\psi$ on all edges of $G$, except possibly those incident with~$v$.

\begin{figure}[htbp]
  \centering
  \includegraphics[height=3.5cm]{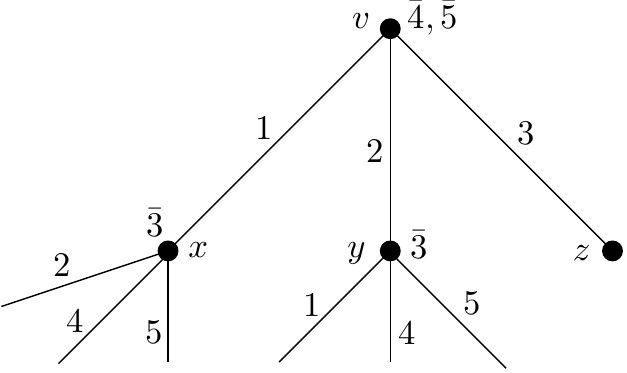}
  \caption{Extending a $K$-change in $G'$ to a $K$-change in $G$ (Lemma \ref{lem:4nonregular})}
  \label{fig:4nonregular1}
\end{figure}

Suppose that after the series of $K$-changes have been applied to $\phi$, we get a colouring $\varphi$ where an edge incident to $v$ (say $xv$)
has different colours under $\varphi$ and $\psi$. We know, by Lemma \ref{lem:path}, that if $\varphi$ and $\psi$ only disagree on a path, then $\phi \sim_5 \psi$, as desired. So, we may assume that $v$ satisfies \ref{sublem:4nonregular:d=3}, and that all three edges incident to $v$ have a colour different in $\varphi$ than in $\psi$.  Suppose that the three edges incident to $v$ are $xv, yv$ and $zv$, and they are coloured 1, 2 and 3 under $\varphi$, respectively. Since every other edge incident to $x, y$ and $z$ agrees in both $\varphi$ and $\psi$, we know that the
colour of each edge in $\psi$ is missing at $x, y$ and $z$, respectively, in $\varphi$. Since, $4$ and $5$ are
missing at $v$, this means that if any of the three edges are coloured 4 or 5 in $\psi$, we can
immediately fix that edge and hence complete our proof by Lemma \ref{lem:path}.  So, we may assume,
without loss of generality, that the colours of $xv$, $yv$ and $zv$ in $\psi$ are $2, 3, 1$,
respectively, under $\psi$ (see Figure \ref{fig:4nonregular2}). Since $v$ satisfies \ref{sublem:4nonregular:d=3}, at least one of $x,y,z$ is missing more than one colour under $\varphi$. Without loss of generality, suppose $x$ has a second missing colour.  If this colour is $4$ or $5$,
then we can recolour $xv$ in $\varphi$, and then recolour $zv$ (both of these are $K$-changes), to complete our proof by applying Lemma \ref{lem:path}.
So, the second missing colour at $x$ under $\varphi$ is $3$. However, now we can make an exchange along
the $(1,3)$-alternating path $xvz$ in $\varphi$, and complete the proof by applying Lemma $\ref{lem:path}$.
\end{proof}

\begin{figure}[htbp]
  \centering
  \includegraphics[height=2.7cm]{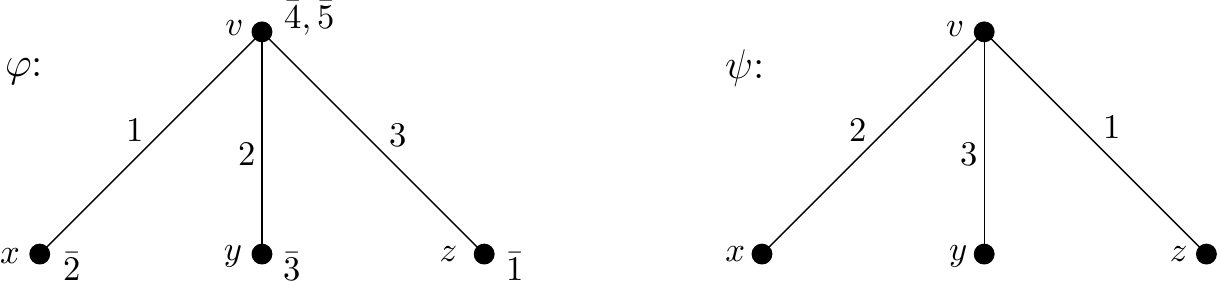}
  \caption{The neighbourhood of $v$ under $\varphi$ and $\psi$ (Lemma \ref{lem:4nonregular})}
  \label{fig:4nonregular2}
\end{figure}

We require two additional preliminary results before completing the proof of Theorem \ref{thm:subquartic}. The first of these is a simple observation about maximum matchings.

\begin{lemma}\label{lem:matching}
  Let $M$ be a maximum matching in a graph $G$, and let $xy$ be an edge of $M$. Suppose that $y$ has a neighbour $y_1$ that is not covered by $M$. Then every neighbour of $x$ that is different from $y_1$ is covered by $M$.
\end{lemma}

\begin{proof}
Suppose, for a contradiction, that $x$ has a neighbour $x_1\neq y_1$ not covered by $M$. Then $x_1x y y_1$ is an augmenting path, contradicting the maximality of $M$.
\end{proof}

The second preliminary result required is in fact substantial, but its proof follows from known proofs of Vizing's Theorem (e.g. \cite{Vi}). The details of this implication are included in \cite{Mo}, and we omit them here.

\begin{lemma}\label{lem:deltaplusone}
  If $G$ is a graph and $c$ is a $k$-edge-colouring of $G$ with $k \geq \Delta+1$, then there exits a $(\Delta+1)$-edge-colouring $c'$ of $G$ such that $c\sim_k c'$.
\end{lemma}

We are now able to prove Theorem \ref{thm:subquartic}.

\setcounter{tmptheorem}{\value{theorem}}
\setcounter{theorem}{\value{thmsubquartic}}
\begin{theorem}
  If $G$ is a graph with maximum degree $\Delta \leq 4$, then $\kappa_E(G, \Delta+2)=1.$
\end{theorem}
\setcounter{theorem}{\value{tmptheorem}}

\begin{proof}
The proof is easy for $\Delta\leq 2$, and for $\Delta=3$ it follows from the same arguments as used below. (We can also use Theorem \ref{thm:subcubic} combined with Lemma \ref{lem:deltaplusone}.) Thus, we shall henceforth assume that $\Delta=4$. Let $\phi$ and $\psi$ be $6$-edge-colourings of $G$. By Lemma \ref{lem:deltaplusone}, there exist $5$-edge-colourings $\phi'$ and $\psi'$ of $G$ such that $\phi\sim_6\phi'$ and $\psi\sim_6\psi'$. Let $M$ be a maximum matching of $G$. Then $\phi'\sim_6\phi''$ and $\psi'\sim_6\phi''$, where $\phi''$, $\psi'$ are obtained from $\phi'$, $\psi'$ by recolouring every edge of $M$ with a new colour $6$. Let $G'=G\setminus M$. Then $\phi''|_{G'}$ and $\psi''|_{G'}$ are both $5$-edge-colourings of $G'$. If we can show that $\phi''|_{G'}\sim_5\psi''|_{G'}$, then this will complete the proof.

Note that there is no pair of adjacent degree 4 vertices in $G'$, as this would correspond to neither vertices being covered by $M$, and $M$ is maximum. So, by Lemma \ref{lem:4nonregular}, we get our desired result provided we can show that, for each connected component of $G'$, there is a vertex $v$ satisfying either: (a) $\deg_{G'}(v)\leq 2$; or (b) $\deg_{G'}(v)=3$ and $v$ has at most two neighbours of degree 4 in $G'$. In fact, we need only show that there exists some maximum matching $M$ of $G$ for which such vertices exist. We shall prove this by induction on the number of components of $G'$.

If $G'$ has only one component, then by Lemma \ref{lem:matching}, there is a vertex $v$ that is covered by $M$, and so are all but at most one of its neighbours. If $v$ has degree at most 3 in $G$, then $v$ has degree at most 2 in $G'$, and hence satisfies (a) in $G'$. Otherwise, $v$ has degree 4 in $G$ with no degree 4 neighbours in $G$. Hence, in $G'$, $v$ has degree 3 with no degree 4 neighbours, and thus satisfies (b).

We may now assume that $G'$ has more than one component. Suppose that $C_1$ is a component of $G'$ such that each vertex in $C_1$ satisfies neither (a) nor (b). Then the minimum degree in $C_1$ is three, so $C_1$ contains a cycle. Since $M$ is maximal, if $xy$ is an edge of this cycle, then exactly one of $x,y$ (say $y$) is covered by $M$ in $G$ (say $yz \in M$). If $z \in V(C_1)$, then one of $y$ or $z$ satisfies (a) or (b), by Lemma \ref{lem:matching} and our above reasoning. So, it must be the case that $z\in V(C_2)$ where $C_2$ is another component of $G'$. Now, define $M':=M -yz+xy$. Then $M'$ is another maximum matching of $G$. Moreover, in $G\setminus M'$, $V(C_1)\cup V(C_2)$ is the vertex set of a single component (since $xy$ is not an edge cut for $C_1$). So, we can complete our proof by induction.
\end{proof}

We conclude this section with the proof of our main result.

\setcounter{tmptheorem}{\value{theorem}}
\setcounter{theorem}{\value{thmsubcubic}}
\begin{theorem}
  If $G$ is a graph with maximum degree $\Delta\leq 3$,  then all $(\Delta+1)$-edge-colourings of $G$ are Kempe equivalent.
\end{theorem}
\setcounter{theorem}{\value{tmptheorem}}

\begin{proof}
\setcounter{claim}{0}
If $\Delta\leq 2$ the claim is easy to verify, so we assume $\Delta=3$. Let $\phi$ and $\psi$ be 4-edge-colourings of a subcubic graph $G$. By Lemma \ref{lem:3nonregular}, we may assume that $G$ is in fact cubic. Let $C$ be a shortest cycle of $G$. For ease of notation, fix an orientation of $C$, and for every vertex $v$ on $C$, let $v_{\minus}$, $v_{\plus}$ and $v_{\circ}$ denote the neighbours of $v$ to the left on $C$, to the right on $C$, and off of $C$, respectively (note that $v_{\circ} \not\in V(C)$ since $C$ is chordless). Let $\phi_v^{\minus}:=\phi(v_{\minus}v)$,  $\phi_v^{\plus}:=\phi(vv_{\plus})$, and $\phi_v^{\circ}:=\phi(vv_{\circ})$, and similarly for $\psi.$

Consider the graph $G'$ obtained from $G$ by deleting the edges of $C$. Every component of $G'$ has a vertex of degree at most 2, so by Lemma \ref{lem:3nonregular} (and our above comments about the case $\Delta\leq 2$), we obtain $\phi |_{G'} \sim_4 \psi |_{G'}$. Let $t$ be the minimum number of Kempe changes in $G'$ required to transform  $\phi |_{G'}$ to $\psi |_{G'}$. We will show that $\phi \sim_{4} \psi$ by induction on $t$.

Suppose first that $t=0$. Then no Kempe changes are required to transform $\phi |_{G'}$ to $\psi |_{G'}$, that is, $\phi$ and $\psi$ agree on every edge of $G$, except possibly the edges of $C$. If at least one edge of $C$ has the same colour under $\phi$ and $\psi$, then we can complete the proof by applying Lemma \ref{lem:path}.  Hence, for every $v$ on $C$, we may assume that  $\phi_v^{\circ}=\psi_v^{\circ}$, $\phi_v^{\minus}\neq \psi_v^{\minus}$ and $\phi_v^{\plus}\neq\psi_v^{\plus}$.  Since we are only working with 4 colours, this immediately tells us that the sets $\{\phi_v^{\minus}, \psi_v^{\minus}\}$ and $\{\phi_v^{\plus}, \psi_v^{\plus}\}$ cannot be disjoint, i.e., they have at least one colour in common. If they have exactly one colour in common, i.e., if the two sets are different, we'll say that $v$ is a \emph{difference} vertex; let $D$ be the set of all difference vertices on $C$ (with respect to $\phi$ and $\psi$).

Our proof of the case $t=0$ now proceeds by induction on $|D|$. If $|D|=0$, then only a single pair of colours appear on $C$ under both $\psi$ and $\varphi$, and hence it is possible to do a $K$-change on $C$ so that the two colourings agree. It is also clear that $|D|\neq 1$, so we assume that $|D| \geq 2$. The following claim is central to our argument that $|D|$ can be reduced.

\begin{claim}\label{claim:balance}
  We may assume that for every vertex $v\in V(C)$, $\psi_v^{\minus}=\phi_v^{\plus}$.
\end{claim}

\begin{proofc}
Consider a pair of difference vertices $x \neq y$ on $C$ such that the path $xx_{\plus} \cdots y$ contains no other difference vertices. Then the edges of this path alternate in colour under $\phi$ (say in colours 1 and 2), and also alternate in 1 and 2 under $\psi$ (with every edge receiving different colours under $\phi$ and $\psi$). Since $x$ and $y$ are difference vertices, we know that exactly one of $\phi_x^{\minus}, \psi_x^{\minus}$ is not 1 or  2, and similarly for $\phi_y^{\plus}, \psi_y^{\plus}$. We claim  that if $\phi_x^{\minus}$ is not 1 or 2, then $\phi_y^{\plus}$ must be 1 or 2, and if $\psi_x^{\minus}$ is not 1 or 2, then $\psi_y^{\plus}$ must be 1 or 2. This is because otherwise, we get that $xx_{\plus}\cdots y$ is a maximal $(1,2)$-alternating path under either $\phi$ or $\psi$, and by the $(1,2)$-exchange along this path we get the two colourings to agree on at least one edge of $C$. So, either $\psi_x^{\minus}=\phi_x^{\plus}$ and $\psi_y^{\minus}=\phi_y^{\plus}$ (in the case when $\phi_x^{\minus} \neq 1,2$), or $\phi_x^{\minus}=\psi_x^{\plus}$ and $\phi_y^{\minus}=\psi_y^{\plus}$ (in the case when $\psi_x^{\minus}\neq 1,2$). Of course, for every vertex $v\neq x,y$ in $xx_{\plus} \cdots y$, both of these properties hold. Moreover, we may repeat this argument for $yy_{\plus} \cdots z$, where $z$ is the next difference vertex to the right of $y$, and so on. So, we may assume that we either have $\psi_v^{\minus}=\phi_v^{\plus}$ for all $v \in V(C)$, or $\phi_v^{\minus}=\psi_v^{\plus}$ for all $v \in V(C)$. Without loss of generality we assume that the former is true.
\end{proofc}

We refer to the above property by saying that each  vertex $v\in V(C)$ is \emph{balanced}.
Note that, since $\phi^\circ_v=\psi^\circ_v\neq\psi^{\plus}_v=\psi^{\minus}_{v_{\plus}}$ and $v_{\plus}$ is balanced for any $v\in V(G)$, we have $\phi^\circ_v\neq\phi^{\plus}_{v_{\plus}}$ and, therefore, $v_\circ vv_{\plus}v_{\plus\plus}$ cannot be an alternating path under $\phi$. Similarly, $v_\circ vv_{\minus}v_{\minus\minus}$ cannot be an alternating path under $\psi$, because $v_{\minus}$ is balanced. This leads to some additional information regarding missing colours on $C$.

\begin{claim}\label{claim:pathbalance}
  The colour $\phi_v^{\minus}$ cannot be missing at $v_{\plus}$ in $\phi$ for any $v\in V(C)$, nor can the colour $\psi_v^{\plus}$ be missing at $v_{\minus}$ in $\psi$ for any $v\in V(C)$.
\end{claim}

\begin{proofc}
If $\phi_v^{\minus}$ is missing at $v_{\plus}$ in $\phi$ for some $v\in V(C)$, then there is a maximal $(\phi_v^{\minus}, \phi_v^{\plus})$-alternating path in $\phi$ beginning at $v_{\plus}$. By the above comment, this path cannot leave $C$. Hence, we may make an exchange along this path. Since $v$ was balanced, this exchange will result in the two colourings agreeing on $v_{\minus}v$. A parallel argument shows that we can get agreement on $vv_{\plus}$ if $\psi_v^{\plus}$ is missing at $v_{\minus}$ in $\psi$ for any $v\in V(C)$.
\end{proofc}

Claims \ref{claim:balance} and \ref{claim:pathbalance} allow us to prove the following.

\begin{claim}\label{claim:doubleswitch}
  For every vertex $v \in V(C)$, there exist $a,b\in\{1,2,3,4\}$ such that $v_{\circ} v v^{\plus}$ is a maximal $(a,b)$-alternating path under $\phi$, and  $v_{\circ}v v_{\minus}$ is a maximal $(a,b)$-alternating path under $\psi$ (see Figure \ref{fig:doubleswitch}).
\end{claim}

\begin{figure}[htbp]
  \centering
  \includegraphics[height=2.5cm]{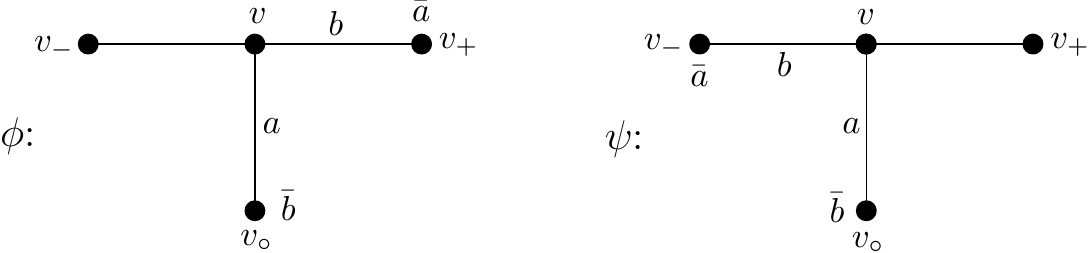}
  \caption{A pictorial representation of Claim \ref{claim:doubleswitch}}
  \label{fig:doubleswitch}
\end{figure}

Before we prove Claim \ref{claim:doubleswitch}, let us see how it will allow us to reduce the number of difference vertices on $C$, and hence complete the proof of the case $t=0$. For any vertex $v$ on $C$, call the operation of simultaneously switching along $v_{\circ} v v_{\plus}$ in $\phi$, and along $v_{\circ} v v_{\minus}$ in $\psi$, a \emph{$v$-double-switch}. Note that after a $v$-double switch, the two colourings still agree everywhere off of $C$ and nowhere on $C$ (and all vertices on $C$ remain balanced). Moreover, note that $v$ will be a member of $D$ after the double-switch if and only if it was a member before the double-switch. On the other hand, note that each of $v_{\minus}$ and $v_{\plus}$ will change their memberships in $D$ after the double-switch.
We claim that there is a path $v_1 \cdots v_r$ of even length on $C$ such that $v_1, v_r$ are two difference vertices. If $|D|\geq 3$ this is clearly true. If $|D|=2$ and if $v_1 \cdots v_r$ is a path on $C$ connecting the only two difference vertices $v_1,v_r$, then we have $\phi^{\plus}_{v_1}\neq\phi^{\minus}_{v_r}$ since every vertex on $C$ is balanced. Hence $r$ is even. In both cases, a double-switch at all $v_i$ with $i$ odd makes $v_0$ and $v_{r}$ non-difference vertices without changing the membership in $D$ of any other vertex on $C$. This reduces the number of difference vertices by two. Hence, by induction on $|D|$, Claim \ref{claim:doubleswitch} will complete our proof of the case $t=0$.

\begin{proofcn}{claim:doubleswitch}
Choose any $v \in V(C)$. We already know that there exist $a,b\in\{1,2,3,4\}$ such that $\phi_v^{\circ}=\psi_v^{\circ}=a$, and $\psi_v^{\minus}=\phi_v^{\plus}=b$, since $v$ is balanced. Without loss of generality let us assume that $a=1$ and $b=2$, and further that $\phi_{v}^{\minus}=3$ and $\psi_v^{\plus}=c$, where $c$ is either 3 or 4, depending on whether or not $v$ is a difference vertex (see Figure \ref{fig:xcolours}). We must show that $v_{\plus}$ is missing colour $1$ in $\phi$, $v_{\minus}$ is missing colour $1$ in $\psi$, and $v_{\circ}$ is missing colour 2 in both colourings.

\begin{figure}[htbp]
  \centering
  \includegraphics[height=2.5cm]{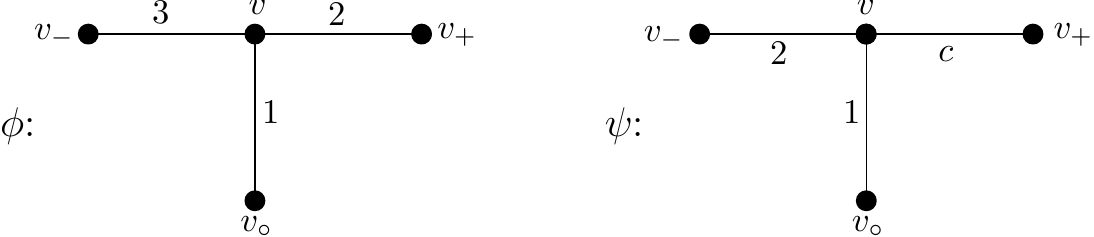}
  \caption{The neighbourhood of $v$ under $\phi$ and $\psi$ in the proof of Claim \ref{claim:doubleswitch}}
  \label{fig:xcolours}
\end{figure}

Note that since $v_{\minus}$  and $v_{\plus}$ are both balanced, 3 cannot be missing at $v_{\minus}$ under $\psi$, nor can $c$ be missing at $v_{\plus}$ under $\phi$. Since 4 is missing at $v$ under $\phi$, we know that 4 cannot be missing at $v_{\plus}$ under $\phi$, as otherwise we could recolour the edge $vv_{\plus}$ in $\phi$ and cause $v$ to be unbalanced. Similarly, since the colour in $\{3,4\}\setminus \{c\}$ is missing at $v$ under $\psi$, this colour cannot be missing at $v_{\minus}$ under $\psi$, as otherwise we could recolour the edge $v_{\minus}v$ in $\psi$ and cause $v$ to be unbalanced. Hence, if $c=3$, then we have succeeded in showing that $1$ is missing at $v_{\plus}$ in $\phi$ and at $v_{\minus}$ in $\psi$. If $c=4$, then to get this conclusion it remains to show that $3$ cannot be missing at $v_{\plus}$ in $\phi$, and $4$ cannot be missing at $v_{\minus}$ in $\psi$.  However, this case is handled by Claim \ref{claim:pathbalance}. Hence we have shown that $v_{\plus}$ is missing colour $1$ in $\phi$, and $v_{\minus}$ is missing colour $1$ in $\psi$.

Suppose now that $d \in \{3,4\}$ is missing at $v_{\circ}$ (in both colourings). Consider the maximal $(1,d)$-alternating paths beginning at $v_{\circ}$ in both $\phi$ and $\psi$. We claim that none of these paths have any edges off of $C$, except for the edge $vv_{\circ}$ itself. In $\phi$, if $d=4$ the path in question ends at $v$, and similarly, in $\psi$, if $d\neq c$ then the path in question ends at $v$. Note that since $v_{\minus}$ is missing the colour $1$ in $\psi$, the edge $v_{\minus}v_{\minus\circ}$ cannot be coloured 1 in either $\psi$ or $\phi$. So in $\phi$, if $d=3$ and the path in question does not end at $v_{\minus}$, then the edge $v_{\minus}v_{\minus\minus}$ has colour 1 under $\phi$, and (because $v_{\minus}$ is balanced) colour 3 under $\psi$. Since the next difference vertex to the left of $v_{\minus}$ must be balanced, we know that either colour 1 or 3 will be missing at this point under $\phi$, and hence the path ends here. We can argue similarly in $\psi$ for the case $d=c$, using the fact that $v_{\plus}$ is missing the colour $1$ in $\phi$. Hence, regardless of the value of $d$, we can simultaneously switch that maximal $(1,d)$-alternating path beginning at $v_{\circ}$ in both $\phi$ and $\psi$, and still get that the two colourings agree on every edge off of $C$.

If $v$ is a difference vertex (i.e., if $c=4$), then of the two maximal $(1,d)$-alternating paths starting at $v_{\circ}$ in $\phi$ and $\psi$, exactly one ends at $v$, and the other one ends at a
vertex $v' \neq v$ on $C$ (as shown above). This implies that after making the two switches induced by these paths, the vertex $v'$ will be unbalanced. This contradiction implies that when $v$ is a difference vertex, $v_{\circ}$ is indeed missing the colour $2$.

If $v$ is not a difference vertex (i.e., if $c=3$), then the maximal $(1,d)$-alternating paths starting at $v_{\circ}$ in $\phi$ and $\psi$ either both end at $v$ (if $d=4$), or they end at vertices $v_{\phi}$, $v_{\psi}$ on $C$, respectively (if $d=3$). If both paths end at $v$, then after switching on both paths the new colour $4$ of $vv_{\circ}$ is not missing at $v_{\minus}$ in $\psi$ (or $v_{\plus}$ in $\phi$), contradicting our argument above.  So, we may assume that $d=3$. It is clear that $v _{\phi}\neq v_{\psi}$.  If the two paths have no edges in common on $C$, then after making both switches, neither $v_{\phi}$ nor $v_{\psi}$ will be balanced. If, on the other hand, the two paths  have common edges then, since $v_{\phi}$ and $v_{\psi}$ are balanced, we get $v_{\phi}=v_{\plus}$, $v_{\psi}=v_{\minus}$, and $v_{\minus}$ and $v_{\plus}$ are both difference vertices. This implies, in particular, that $v_{\minus}$ is a difference vertex. So, by our argument above, we may perform a double-switch at $v_{\minus}$. After this double-switch, we get that $3$ is missing at $v$ under the modified colouring $\phi$, so the maximal $(1,3)$-alternating path beginning at $v_{\circ}$ now ends at $v$. As we have already discussed this case, we get our desired result. This completes the proof of Claim \ref{claim:doubleswitch} and hence we complete the proof of the base case when $t=0$.
\end{proofcn}

We may now assume that $t\geq 1$. Then there is a nonempty series of 4-edge-colourings $\phi_0, \phi_1, \ldots, \phi_t$ of $G'$ such that $\phi_0=\phi |_{G'}$, $\phi_t= \psi |_{G'}$, and  $\phi_i$ and $\phi_{i+1}$ differ by a single Kempe change, for all $i \in \{0, 1, \ldots, t-1\}$. Consider the 4-edge-colouring $\phi_{1}$ of $G'$. Since only $t-1$ Kempe changes are required to transform $\phi _1$ to $\psi |_{G'}$ in $G'$, we know by induction that $\varphi \sim_{4} \psi$ for $\varphi$ any 4-edge-colouring of $G$ with $\varphi |_{G'}=\phi_1$. Thus it suffices to show that there exists such a $\varphi$ where $\phi \sim_{4} \varphi$.

Consider the Kempe change which transforms $\phi |_{G'}$ to $\phi_1$ in $G'$. Let $H'$ be the maximal alternating component in $\phi |_{G'}$ corresponding to this exchange, and let $H$ be the maximal alternating component of $\phi$ containing $H'$.  If  $E(H) \setminus E(H') \subseteq E(C)$ then we get our desired $\varphi$ from $\phi$ simply by swapping along $H$.  So, suppose that this is not the case. Then $H'$ is a path, and $H$ contains a path $w_{\circ} w w_{\plus} \cdots x x_{\circ}$ for some pair of vertices $w, x$ on $C$, where one of $w$ or $x$ is an endpoint of $H'$ in $G'$. The second endpoint of $H'$ may or may not be on $C$, and there may or may not be a segment $w'_{\circ} w' w'_{\plus} \cdots x' x'_{\circ}$ of $H$ where $w'$ or $x'$ is this second endpoint. We will first show how to resolve the case where the second endpoint is not on $C$, and then show how our argument can be extended to resolve the other cases as well.

Suppose, without loss of generality, that $H$ is an alternating $(1,2)$-component in~$\phi$. Then, it is easy to see that  there exists  a maximal $(1,2)$- or $(3,4)$-alternating path $y y_{\plus} \cdots z$ on $C$ that is edge-disjoint from $w w_{\plus} \cdots x$. In fact, we can make the following more general observation that will later be of use.

\begin{obs}\label{obs:pathbetween}
  Let $u$ and $v$ be vertices on $C$ and let $a=\phi_u^{\minus}$ and $b=\phi_u^{\circ}$ ($a,b \in \{1, 2, 3, 4\}$). Suppose that $\{\phi_v^{\circ}, \phi_v^{\plus}\}$ is either this same pair of colours $\{a,b\}$ or their complement $\{c,d\}=\{1,2,3,4\}\setminus\{a,b\}$. Then there exists a maximal $(a,b)$- or $(c,d)$-alternating path $p p_{\plus} \cdots q$ in $\phi$ that lies entirely on $C$ and is edge-disjoint from $v v_{\plus}\cdots u$ (cf. Figure \ref{fig:pathbetween}).
\end{obs}

\begin{figure}[htbp]
  \centering
  \includegraphics[height=2.2cm]{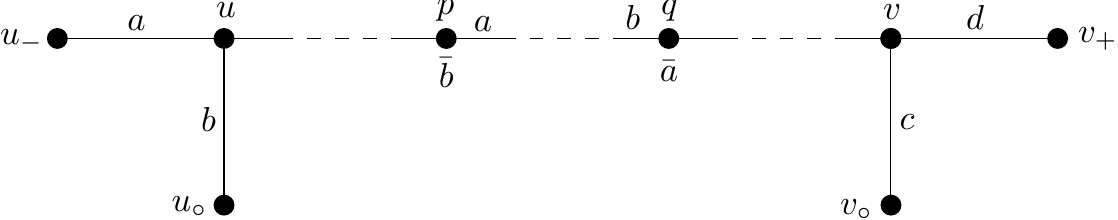}
  \caption{An instance of Observation \ref{obs:pathbetween}}
  \label{fig:pathbetween}
\end{figure}

Our argument will now focus exclusively on the segment $w w_{\plus} \cdots x\cdots y y_{\plus} \cdots z$ of $C$. We shall show, by induction on the distance between $x$ and $y$, that we can resolve the present case by making $K$-changes involving only the edges in this segment. First, note that every vertex in $w w_{\plus} \cdots x$  is missing one of $3$ or $4$ under $\phi$. If any adjacent pair is missing the same colour, then by simply recolouring the edge between them we get our desired result. Hence, we may assume that each adjacent pair of vertices in this path are missing different colours. If $x=y$, then making the $K$-change along the path $y y_{\plus} \cdots z$ will cause $x_{\minus}$ and $x$ to be missing the same colour. Hence, we may assume that $x\neq y$.

Assume, without loss of generality, that $\phi_x^{\minus}=1$ and that colour 3 is missing at $x_{\minus}$. Let $L$ be the maximal $(1,3)$-alternating path beginning at $x_{\minus}$. Note that, by possibly switching along $y y_{\plus} \cdots z$, we can ensure that $\{\phi_y^{\circ}, \phi_y^{\plus}\}=\{1,3\}$ or $\{2,4\}$. Then, the last vertex of $L$ on $C$, say $u$, must be to the left of $y$. If $L$ does not leave $C$, then we may switch along $L$. After this switch, the maximal $(1,2)$-alternating component containing the edge $ww^{\circ}$ has $x_{\minus}$ as an endpoint. On the other hand, the maximal $(1,2)$-alternating component containing $xx_{\circ}$ also contains $xx_{\plus}$. So, the only way that we do not have our desired result is if the endpoint of $H'$ is $x$, and there exists a vertex $\bar{x}$ on $C$ such a that $x_{\circ} x x_{\plus} \cdots \bar{x} \bar{x}_{\circ}$ is a $(1,2)$-alternating path. However, in this scenario, we get that the distance between $\bar{x}$ and $y$ is less than between $x$ and $y$, hence we complete the proof by induction. So, we may assume that $L$ leaves $C$. But then, by Observation \ref{obs:pathbetween}, there exists a maximal $(1,3)$- or $(2,4)$-alternating path $p p_{\plus} \cdots q$ on $C$ which is edge-disjoint from $y y_{\plus} \cdots x$. Now, by possibly switching on $p p_{\plus} \cdots q$, we can ensure that $\{\phi_p^{\circ}, \phi_p^{\plus}\}=\{1,2\}$ or $\{3,4\}$. Consequently, Observation \ref{obs:pathbetween} guarantees a maximal $(1,2)$- or $(3,4)$-alternating path on $C$ which is edge-disjoint from $p p_{\plus} \cdots x$. As this path is closer to $x$ than $y y_{\plus} \cdots z$ is, we complete the inductive step.

Note that our above proof works identically well if instead of concentrating on the segment $w w_{\plus} \cdots x\cdots y y_{\plus} \cdots z$ of $C$, we concentrated on the segment $y y_{\plus} \cdots z \cdots \allowbreak w w_{\plus} \ldots x$.  The only adjustment that would need to be made would be to interchange the terms left and right in our argument. This symmetry is important for dealing with our remaining cases.

Suppose now that the second endpoint $v$ of $H'$ is on $C$. We first deal with the case that $v$ is not $w'$ or $x'$ for any $w'_\circ w' w'_{\plus} \cdots x' x'_\circ$ that is a part of $H$. Choose a maximal $(1,2)$- or $(3,4)$-alternating path on $C$ that is edge-disjoint from $x x_{\plus} \cdots w$, as guaranteed by Observation \ref{obs:pathbetween}. There are two segments of $C$ to which we may apply our symmetric proof --- if one of these does not contain $v$, then simply apply the proof to this segment. If both these segments of $C$ contain $v$, then $v$ must occur on the maximal alternating path. If it is an endpoint of the path, then $v$ is not an endpoint of $H$, and we may apply our proof to the segment of $C$ that ends at $v$. Hence, we may assume that $v$ is somewhere in the middle of our maximal alternating path. Then, $v$ must be an endpoint of $H$, missing either 1 or 2. Looking at our above argument however, we can see that none of the exchanges we might do would change that fact that 1 or 2 is missing at $v$.

It remains to address the case where the second endpoint $v$ of $H'$ is $w'$ or $x'$, where $w'_\circ w' w'_{\plus} \cdots x' x'_\circ$ is a part of $H$. Then, by Observation \ref{obs:pathbetween}, there is a maximal $(1,2)$- or $(3,4)$-alternating path on $C$ that is edge disjoint from $w'w_{\plus}\cdots x'$. We can use this path to apply our argument for $w'_\circ w' w'_{\plus} \cdots x' x'_\circ$. Since no edge on $w_\circ w w_{\plus} \cdots x x_\circ$ is recoloured during this procedure, this reduces the problem to the case above. This completes the $t\geq1$ inductive step, and hence completes the proof.
\end{proof}

\end{section}

\begin{section}{Multigraphs}

While Vizing's Theorem classifies all graphs as either class 1 or class 2, \emph{multigraphs} (i.e. graphs with multiple edges) can have chromatic index much larger than $\Delta+1$. The main results of this paper do extend to multigraphs however, and in fact, the proofs simplify greatly when multiple edges exist.

\setcounter{theorem}{10}
\begin{theorem}\label{thm:multigraphs}
  Let $G$ be a multigraph with maximum degree $\Delta$.
  \begin{enumerate}[label=\textup{(\alph*)}]
  \item\label{subthm:multigraphs:Delta<=3}
    If $\Delta \leq 3$, then all $(\Delta+1)$-edge-colourings of $G$ are Kempe equivalent.
  \item\label{subthm:multigraphs:Delta=4}
    If $\Delta=4$, then all 6-edge-colourings of $G$ are Kempe equivalent.
  \end{enumerate}
\end{theorem}

\begin{proof}
We prove both \ref{subthm:multigraphs:Delta<=3} and \ref{subthm:multigraphs:Delta=4} by induction on $|E(G)|$. If $G$ has no multiple edges then the result follows from Theorems \ref{thm:subquartic} and \ref{thm:subcubic}. Hence we may assume that there exist edges $e,f \in E(G)$, and both $e$ and $f$ have endpoints $x$ and $y$. Let $\phi$ and $\psi$ be $k$-edge-colourings of $G$, where $k=\Delta+1$ if $\Delta\leq 3$ and $k=6$ if $\Delta=4$. Let $G':=G\setminus e$, and observe that by induction we get $\phi |_{G'} \sim_{k} \psi |_{G'}$. We claim that this series of $K$-changes in $G'$ extends to a series of $K$-changes in $G$. Note that this claim is sufficient to complete the proof, as it implies that there is indeed a series of $K$-changes which can be performed on $\phi$ so that it agrees with $\psi$ on all edges of $G'$. Since $e$ is just a path, we may then apply Lemma \ref{lem:path} to get complete agreement of the colourings.

Suppose there is a maximal $(1,2)$-alternating path which ends at $x$ in some intermediate colouring $\varphi$ of $G'$, but which includes the edge $e$ and at least one other edge incident to $y$ in $G$. Note that $f$ cannot belong to this path. Suppose that $\varphi(f)=3$. If $\Delta=3$, we see that colour 4 is missing at both $x$ and $y$ in $\varphi$, and hence we may break the $(1,2)$-path by recolouring $e$ with $4$.  On the other hand, if $\Delta=4$, then at least two of the colours $4,5,6$ are missing at $x$, and at least two of these colours are missing at $y$. Hence $x$ and $y$ again have a common missing colour, and this can again be used to recolour $e$ and hence break the $(1,2)$-path. This completes the proof.
\end{proof}

\end{section}

\end{document}